\documentclass[11pt]{article}

\usepackage{amssymb}
\usepackage{amsmath}
\usepackage[dvips]{graphics}
\usepackage{epsfig}
\usepackage{arydshln}
\usepackage{amsthm}
\usepackage{pgfplots}
\usepackage{arydshln}

\usepackage[left=24mm, textwidth=160mm]{geometry}

\theoremstyle{plain}
\newtheorem{theorem}{Theorem}
\newtheorem{corollary}[theorem]{Corollary}
\newtheorem{lemma}[theorem]{Lemma}

\theoremstyle{definition}
\newtheorem{definition}[theorem]{Definition}
\newtheorem{example}[theorem]{Example}

\theoremstyle{remark}
\newtheorem{remark}[theorem]{Remark}

\begin{document}

%\date{}

\title{The Fibonacci--Redheffer Matrix and its Properties}

\author{Aristides V.\ Doumas\thanks{Department of Mathematics,
        School of Applied Mathematical and Physical Sciences,
        National Technical University of Athens, Zografou Campus,
        15780 Athens, Greece (adou@math.ntua.gr, \,ppsarr@math.ntua.gr).}
        \thanks{Archimedes/Athena Research Center, Greece.}\,
        and
        Panayiotis J.\ Psarrakos\footnotemark[1]}

\date{}

\maketitle

\vspace{-6mm}

\vspace{2mm}

\begin{abstract}
In this paper, we introduce a Redheffer--type matrix with Fibonacci entries and study its determinant and spectral properties.
In particular, we derive its determinant and show that it is always negative, prove that its eigenvalues are real and simple and,
except for the smallest eigenvalue, lie close to Fibonacci numbers, and establish that every entry of each eigenvector is nonzero.
We also consider more general Redheffer--type matrices, present intriguing examples, and discuss several asymptotic results.
\end{abstract}


\textbf{MSC 2020 Mathematics Classification.}   15A18,
                                                15B36,
                                                11B39,
                                                11C20.

%%%%%%%%%%%%%%%%%%%%%%%%%%%%%%%%%%%%%%%%%%%%%%%%%%%%%%%%%%%%%%%%%%%%%%%%%%%%%%%%%%%%%%%%%%%%%%%%%%%%%%%%%%%%%%%%%%%%
\section{Introduction} \label{section1}

In 1977, Redheffer\footnote{Raymond Moos Redheffer was a mathematician-engineer
who also built an electronic device for the famous game of \textit{Nim} demonstrating how mathematics can become a thinking machine.}
\cite{R1} introduced the Redheffer matrix and defined it as the $n \times n$ real matrix $R_n = \bigl [ r_{i,j} \bigr ]$ such that,
for $i,j=1,2,\ldots,n$,
\[
  r_{i,j} = \begin{cases}
          1, & \text{if $i\mid j$ or $j=1$;} \\
          0, & \text{otherwise.}
       \end{cases}\label{BEG}
\]
Hence, for $n=8$,
\[
  R_8 = \begin{bmatrix}
  1 & 1 & 1 & 1 & 1 & 1 & 1 & 1 \\
  1 & 1 & 0 & 1 & 0 & 1 & 0 & 1 \\
  1 & 0 & 1 & 0 & 0 & 1 & 0 & 0 \\
  1 & 0 & 0 & 1 & 0 & 0 & 0 & 1 \\
  1 & 0 & 0 & 0 & 1 & 0 & 0 & 0 \\
  1 & 0 & 0 & 0 & 0 & 1 & 0 & 0 \\
  1 & 0 & 0 & 0 & 0 & 0 & 1 & 0 \\
  1 & 0 & 0 & 0 & 0 & 0 & 0 & 1
  \end{bmatrix}  .
\]

Redheffer \cite{R1} proved that
\begin{equation}
  \det ( R_n ) = M(n),     \label{1}
\end{equation}
where $M(n)$ is the {\it Mertens function} defined for any positive integer $n$, as
$
     M(n) := \sum\limits_{j=1}^n \mu(j) ,
$
and $\mu(j)$ denotes (as usual) the {\it M{\"o}bius function}, namely,
\[
     \mu(j) = \begin{cases}
              0,       & \text{if $j$ has one or more repeated prime factors;} \\
              1,       & \text{if $j=1$;} \\
              (-1)^{k},& \text{if $j$ is a product of $k$ distinct primes.}
              \end{cases}
\]
Bordell{\`e}s and Cloitre \cite{BOR} provided an alternative proof of (\ref{1}), based upon the $LU$ decomposition of the Redheffer matrix.
Relation (\ref{1}) is of great importance due to its connection with the Riemann hypothesis.
It is well known that the Riemann hypothesis is true if and only if
\begin{equation}
      M(n) = O(n^{1/2 +\varepsilon}) \quad \text{for any $\varepsilon > 0$;}     \label{RH}
\end{equation}
for details, see, e.g., Littlewood \cite{LIT} and Titchmarsh \cite{TIT}.

For an $n \times n$ complex matrix $A$, recall that the \textit{characteristic polynomial} of $A$ is the $n$-th degree monic polynomial
$\chi_A (z)  = \det ( z\, I_n - A )$, where $z$ is a complex variable and $I_n$ denotes the $n \times n$ identity matrix.
The zeros of $\chi_A (z)$ are the \textit{eigenvalues} of the matrix $A$.
We let $\sigma(A) = \bigl \{ \lambda \in \mathbb{C} : \, \chi_A (\lambda) = 0 \bigr \}$ denote the set of all eigenvalues of $A$, which we call the \textit{spectrum} of $A$.
We also define the \textit{spectral radius} of $A$ by $\rho (A) = \max \bigl \{ |\lambda| : \, \lambda \in \sigma(A) \bigr \}$.
For an eigenvalue $\lambda$ of $A$, any nonzero vector $\mathbf{x}$ satisfying $( A - \lambda\, I_n ) \mathbf{x} = \mathbf{0}$ is a \textit{(right) eigenvector} of $A$ corresponding to $\lambda$.
The eigenvectors of the conjugate transpose matrix $A^*$, corresponding to the conjugates of the eigenvalues of $A$, are the \textit{left eigenvectors} of $A$.
For an eigenvalue $\lambda \in \sigma(A)$, the \textit{algebraic multiplicity} of $\lambda$ equals its multiplicity as a zero of the characteristic polynomial $\chi_A (\lambda)$,
and it is always greater than or equal to the \textit{geometric multiplicity} of $\lambda$, namely, the dimension of the null space of $A - \lambda I_n$.
We call this null space the \textit{eigenspace} of $A$ corresponding to $\lambda$.
We refer the reader to the book of Horn and Johnson \cite{HJ} for a comprehensive treatment of matrix theory.

The spectral structure of the Redheffer matrix $R_n$ has attracted the interest of researchers, since thanks to relation (\ref{1}),
$M(n) = \prod\limits_{j=1}^{n} \lambda_{j}$, where $\lambda_1 , \lambda_2 , \ldots , \lambda_n$ are the $n$ (not necessarily distinct) eigenvalues of $R_n$.
We refer the interested reader to \cite{BFP,BJ,CS,J,RB,v1,v2} for well known results on the Redheffer matrix;
at the beginning of Section \ref{section3}, we will present a short overview of some of these results.
To highlight how structured growth influences spectral behavior, we introduce a matrix with the same zero pattern as $R_n$, where the $1$'s in each $i$-th row
except those in the first column below the $(1,1)$-th entry, are replaced by the $i$-th element of the Fibonacci sequence. By choosing the Fibonacci sequence,
which is a positive, rapidly increasing sequence, we aim to emphasize and better understand the resulting spectral differences from the Redheffer matrix.
>From here and in what follows, we adopt the following definition for the Fibonacci numbers $F_1 , F_2 , F_3 , \ldots\,$:
\[
    F_1 = F_2 = 1   \quad\;\; \text{and} \quad\;\;   F_n = F_{n-1} + F_{n-2}   \;\;\;   (n = 3 , 4 , \ldots) .
\]

\begin{definition}   \em
For $n \geq 3$, we define the $n \times n$ \textit{Fibonacci--Redheffer matrix} $F_R(n) = \bigl [ F_R(i,j) \bigr ]$ such that, for $i,j=1,2,\ldots,n$,
\begin{equation}
  F_R(i,j) : = \begin{cases}
          1,     & \text{if $j=1$;}     \\
          F_{i}, & \text{if $i\mid j$;} \\
          0,     & \text{otherwise.}                    \label{2}
       \end{cases}
\end{equation}
\end{definition}

For example,
\[
F_R(8) =
\left[
\begin{array}{cccc:cccc}
1 & 1 & 1 & 1 & 1 & 1 & 1 & 1 \\
1 & 1 & 0 & 1 & 0 & 1 & 0 & 1 \\
1 & 0 & 2 & 0 & 0 & 2 & 0 & 0 \\
1 & 0 & 0 & 3 & 0 & 0 & 0 & 3 \\ \hdashline
1 & 0 & 0 & 0 & 5 & 0 & 0 & 0 \\
1 & 0 & 0 & 0 & 0 & 8 & 0 & 0 \\
1 & 0 & 0 & 0 & 0 & 0 & 13 & 0 \\
1 & 0 & 0 & 0 & 0 & 0 & 0 & 21
\end{array}
\right].
\]

\begin{remark} \label{firstremark}    \rm
If we look the Fibonacci--Redheffer matrix above as a $2\times 2$ block matrix, then we observe that the upper left block of $F_{R}(n)$ is also a Fibonacci--Redheffer matrix
of size $\bigl \lfloor \frac{n}{2} \bigr \rfloor$ (in this case, we have the matrix $F_{R}(4))$, where we let $\lfloor x \rfloor$ denote the {\it floor function},
i.e., the function that takes as argument a real number $x$ and gives as result the greatest integer less than or equal to $x$.
The lower left block is a matrix with ones in the first column and all the other entries equal to zero,
and the lower right block is a diagonal matrix with diagonal entries the Fibonacci numbers starting from $F_{\lfloor \frac{n}{2} \rfloor+1}$.
\end{remark}

\begin{remark} \label{sparse}    \rm
The $n \times n$ matrices $F_{R}(n)$ and $R_{n}$ are entrywise nonnegative and share the same zero pattern.
Let $S_{n}$ denote the total number of nonzero entries of the matrix $R_{n}$ (and thus, also of $F_{R}(n)$).
 2The definitions of the above two matrices imply that the number of nonzero entries in each column $j$ is given by the {\it divisor function} $d(j)$ for $j \geq 2$.
Hence, we have $S_{n}= n + \sum\limits_{j=2}^{n} d(j)$.
By applying the Dirichlet hyperbola method to the divisor sum, we get (see, e.g., \cite{APOSTOL})
$S_{n} = n \ln n+2 \gamma n + O(\sqrt{n})$, as $n \rightarrow \infty$,
where $\gamma = 0.5772\cdots$ is the \textit{Euler$-$Mascheroni constant}.
\end{remark}

The outline of the rest of the paper follows.
In the next section, we derive the determinant of the Fibonacci--Redheffer matrix $F_R (n)$ for all positive integers $n$, and we observe that it is always negative.
We also study the asymptotic behaviour of the determinant as $n \rightarrow \infty$.
In Section \ref{section3}, we prove that the eigenvalues of $F_R (n)$ are real and simple, and all the entries of the eigenvectors of $F_R (n)$ are nonzero.
Moreover, the eigenvalues of $F_R (n)$, except the smallest one (which is the only negative eigenvalue), lie close to Fibonacci numbers.
Finally, in Section \ref{section4}, we study the asymptotic behaviour of the determinant of more general Redheffer--type matrices through some intriguing examples.

%%%%%%%%%%%%%%%%%%%%%%%%%%%%%%%%%%%%%%%%%%%%%%%%%%%%%%%%%%%%%%%%%%%%%%%%%%%%%%%%%%%%%%%%%%%%%
\section{The Determinant of $F_R(n)$}  \label{section2}

Let $n!_{F}$ denote the {\it Fibonacci factorial}, which is the product of the first $n$ positive Fibonacci numbers, namely,
\begin{equation} \label{4}
    n!_{F} = \prod_{i=1}^{n}F_{i}.
\end{equation}

\begin{theorem}  \label{firsttheorem}
Consider the Fibonacci--Redheffer matrix defined by (\ref{2}).
If $\lambda_1 , \lambda_2 , \ldots , \lambda_n$ are the $n$ (not necessarily distinct) eigenvalues of $F_R(n)$, then
\[
  \det ( F_R (n) )  =  \prod_{j=1}^{n}\lambda_j =  n!_{F}\sum_{k=1}^{n}\frac{\mu(k)}{F_{k}} .
\]
\end{theorem}

It is worth noting that the Fibonacci--Redheffer matrix $F_R (n)$ is essentially a weighted version of the Redheffer matrix $R_n$,
and as a consequence, its determinant is a weighted sum over the permutations compatible with $R_n$ (i.e., a weighted sum over the M{\"o}bius function).
For the proof of this theorem, we need the following representation of the matrix $F_R (n)$:
\begin{equation}
  F_R (n) = C (n) + D (n) , \label{rep}
\end{equation}
where, for $i=1,2,\ldots,n$ and $j=1,2,\ldots,n$, the entries of the matrix $C(n) = \bigl [ c_{i,j} \bigr ]$ are
\begin{equation}
  c_{i,j} = \begin{cases}
          1, & \text{if $j=1$ and $i \neq 1$;} \\
          0, & \text{otherwise}
       \end{cases} \label{C}
\end{equation}
and the entries of the matrix $D(n) = \bigl [ d_{i,j} \bigr ]$ are
\begin{equation}
  d_{i,j} = \begin{cases}
          F_i, & \text{if $i\mid j$;} \\
            0, & \text{otherwise.}
       \end{cases} \label{D}
\end{equation}
For example,
\begin{equation*}
  F_R(8) =
  C(8) + D(8) =
  \begin{bmatrix}
  0 & 0 & 0 & 0 & 0 &0 &0 &0 \\
  1 & 0 & 0 & 0 & 0 &0 &0 &0 \\
  1 & 0 & 0 & 0 & 0 &0 &0 &0 \\
  1 & 0 & 0 & 0 & 0 &0 &0 &0 \\
  1 & 0 & 0 & 0 & 0 &0 &0 &0 \\
  1 & 0 & 0 & 0 & 0 &0 &0 &0 \\
  1 & 0 & 0 & 0 & 0 &0 &0 &0 \\
  1 & 0 & 0 & 0 & 0 &0 &0 &0
  \end{bmatrix}
  +
  \begin{bmatrix}
  1 & 1 & 1 & 1 & 1 &1 &1 &1 \\
  0 & 1 & 0 & 1 & 0 &1 & 0 &1 \\
  0 & 0 & 2 & 0 & 0 &2 & 0 & 0 \\
  0 & 0 & 0 & 3 & 0 &0 & 0 & 3 \\
  0 & 0 & 0 & 0 & 5 &0 & 0 & 0 \\
  0 & 0 & 0 & 0 & 0 &8 & 0 & 0 \\
  0 & 0 & 0 & 0 & 0 &0 & 13 & 0 \\
  0 & 0 & 0 & 0 & 0 &0 & 0 & 21
  \end{bmatrix} .
\]
Note that the matrix $ C(n)$ is clearly a rank one singular matrix.
The upper triangular matrix $D(n)$ is, of course, nonsingular since
\begin{equation}
    \det (D(n)) = \prod_{j=1}^{n} F_{j} = n!_{F} ,    \label{fibonorial}
\end{equation}
where $n!_{F}$ is the Fibonacci factorial defined by (\ref{4}).

Let us now turn our attention to the matrix $D(n)$. In particular, we will derive its inverse.

\begin{lemma} \label{firstlemma}
The inverse of matrix $D(n)$ is given by the formula
\begin{equation*}
D(n)^{-1} := \bigl [ \tilde{d}_{i,j} \bigr ] , \;\; \text{where} \quad
          \tilde{d}_{i,j} = \begin{cases}
          \dfrac{1}{F_j} \, \mu \biggl ( \dfrac{j}{i} \biggr ), & \text{if $i\mid j$;} \\
          0, & \text{otherwise.}
       \end{cases}
\end{equation*}
\end{lemma}

For example,
\[
  D(8)^{-1}=\begin{bmatrix}
  1 & -1 & -\frac{1}{2} & 0 & -\frac{1}{5} &\frac{1}{8} &-\frac{1}{13} &0\\
  0 & 1 & 0 & -\frac{1}{3} & 0 &-\frac{1}{8} & 0 &0\\
  0 & 0 & \frac{1}{2} & 0 & 0 &-\frac{1}{8} & 0 & 0 \\
  0 & 0 & 0 & \frac{1}{3} & 0 &0 & 0 & -\frac{1}{21}  \\
  0 & 0 & 0 & 0 & \frac{1}{5} &0 & 0 & 0\\
  0 & 0 & 0 & 0 & 0 &\frac{1}{8} & 0 & 0\\
  0 & 0 & 0 & 0 & 0 &0 & \frac{1}{13} & 0\\
  0 & 0 & 0 & 0 & 0 &0 & 0 & \frac{1}{21}\\
  \end{bmatrix} .
\]

\noindent \textit{Proof of Lemma $\ref{firstlemma}$.}\
It suffices to obtain that the product $D(n) D(n)^{-1}$ coincides with the $n \times n$ identity matrix $I_n$.
The $(i,j)$-th entry of the product $D(n)D(n)^{-1}$ is $q_{i,j} = \sum\limits_{k=1}^{n} d_{i,k} \tilde{d}_{k,j}$.
The term $d_{i,k} \tilde{d}_{k,j}$ is equal to $0$, unless $i \mid k$ and $k \mid j$.
Hence, if $i$ does not divide $j$, then we have $q_{i,j}=0$.
On the other hand, if $i \mid j$, then there exists a positive integer $m$ such that $k = m \cdot i$, and since $k \mid j$, we get $m \mid \dfrac{j}{i}$.
For convenience, we set $d := \dfrac{j}{i}$.
Then, we observe that
\[
  q_{i,j} = \sum\limits_{\tiny{\begin{array}{c} i\mid k \\ k\mid j \end{array}}} \dfrac{F_{i}}{F_{j}} \, \mu \biggl ( \dfrac {j}{k} \biggr )
          = \sum_{m \mid \frac{j}{i}} \dfrac{F_{i}}{F_{j}} \, \mu \biggl ( \frac{\frac{j}{i}}{m} \biggr )
          = \dfrac{F_{i}}{F_{j}} \sum_{m \mid d} \mu \biggl ( \dfrac{d}{m} \biggr ) .
\]
Now,
\[
  \sum_{m \mid d}\mu \biggl (\frac{d}{m} \biggr ) = \begin{cases}
          1, & \text{if $d=1$, i.e., if $i=j$;} \\
          0, & \text{otherwise}
       \end{cases}\label{mobius}
\]
(see, e.g., \cite{APOSTOL} for details). Thus,
\[
  q_{i,j} = \begin{cases}
          1, & \text{if $i=j$;} \\
          0, & \text{otherwise,}
       \end{cases}
\]
which completes the proof. $\hfill \square$

Having Lemma \ref{firstlemma}, we proceed with the determinant $\det (F_R (n))$.

\noindent \textit{Proof of Theorem $\ref{firsttheorem}$.}\
By (\ref{rep}), we get
\begin{eqnarray*}
  \det (F_R (n))    &=& \det \bigl (  D(n) D(n)^{-1} F_R (n) \bigr ) \\
                    &=& \det (D(n)) \det \bigl ( D(n)^{-1} ( C(n) + D(n) ) \bigr ) \\
                    &=& \det (D(n)) \det \bigl ( D(n)^{-1} C(n) + I_n \bigr ) .
\end{eqnarray*}
Moreover, the matrix $D(n)^{-1} C(n)$ is lower triangular with all its diagonal entries zero except the first one which is equal to
$\,-1 - \dfrac{1}{2} - \dfrac{1}{5} + \dfrac{1}{8} - \dfrac{1}{13} + \cdots$.
Hence, for the matrix $D(n)^{-1}C(n)+I(n)$, one immediately can see that
\[
   \det \bigl ( D(n)^{-1}C(n) + I_n \bigr ) = \sum_{k=1}^{n} \dfrac{ \mu (k)}{F_k} .
\]
The proof of Theorem \ref{firsttheorem} is completed by invoking relation (\ref{fibonorial}).       $\hfill\square$

The next theorem is entertaining.

\begin{theorem}  \label{secondtheorem}
As $n\rightarrow \infty$, we get the following asymptotic result:
\begin{equation}\label{kou}
  \sum_{k=1}^{n} \dfrac{\mu(k)}{F_{k}} \sim C ,
\end{equation}
meaning (for series) that $\lim\limits_{n\rightarrow \infty} \sum\limits_{k=1}^{n} \dfrac{\mu(k)}{F_{k}} = C$, where
\begin{equation}
   C \approx - 0.64572472\cdots \label{Ca}
\end{equation}
Moreover, for all positive integers $n \geq 3$, we have
\begin{equation}\label{kou1}
  \sum_{k=1}^{n} \dfrac{\mu(k)}{F_{k}} < 0 ,
\end{equation}
i.e., all the partial sums of (\ref{kou}) are negative for $n \geq 3$.
\end{theorem}

\begin{proof}
Let $\phi := \dfrac{1+\sqrt{5}}{2}$ be the golden ratio.
Then, by Binet's formula \cite{KPG}, we clearly get
\[
    F_n \sim \dfrac{\phi^{n}}{\sqrt{5}} \quad \text{as $n\rightarrow \infty$,}
\]
where the symbol $\sim$ has the following meaning (for sequences): $a_n \sim b_n$ as $n\rightarrow \infty$ if and only if $\lim \dfrac{b_n}{a_n}=1$.
The series of relation (\ref{kou}) is absolutely convergent.
Thus, the leading term in its asymptotic expansion is a constant (as we will see, the one mentioning in (\ref{Ca})).
Let us now verify that the series converges actually to a negative number.
It is well known, and easy to check, that for all positive integers $n$,
\begin{equation}
     F_n \geq \phi^{n-2}.       \label{phi}
\end{equation}
The desired verification follows directly by the strong (second) principle of mathematical induction (see, e.g., \cite{DK} where this inequality is stated  as an exercise).

Let $k_0$ be a fixed positive integer. We observe that
\begin{equation}
   \biggl | \sum_{k>k_{0}}^{\infty}\frac{\mu(k)}{F_{k}} \biggr |    \leq    \dfrac{\phi^{2-k_0}}{\phi -1},      \label{kou2}
\end{equation}
where we have used (\ref{phi}) and summed the resulting infinite geometric series.
Having the upper bound of relation (\ref{kou2}), we are able to calculate the partial sums
\begin{equation}
      \sum_{k=1}^{k_{0}}\frac{\mu(k)}{F_{k}}     \label{kou3}
\end{equation}
and prove that they are negative. If, for example, $k_0 =12$, then (\ref{kou2}) implies
\begin{equation}
   \biggl | \sum_{k=13}^{\infty} \dfrac{\mu(k)}{F_{k}} \biggr | <  0.013192 . \label{kou4}
\end{equation}
On the other hand,
\begin{equation}
    \sum_{k=1}^{12}\frac{\mu(k)}{F_{k}} = \frac{1}{1}-\frac{1}{1}-\frac{1}{2}
    +\frac{0}{3}-\frac{1}{5}+\frac{1}{8}-\frac{1}{13}+\frac{0}{21}+\frac{0}{34}
    +\frac{1}{55}-\frac{1}{89}+\frac{0}{144}
    = -0.6449772\cdots,
    \label{kou5}
\end{equation}
and hence, (\ref{kou}) is true.
Now, it is easy to check that all the partial sums of (\ref{kou3}) from $k_0 = 3$ up to $k_0=12$ are negative.\footnote{For $k_0 =2$, the sum is zero.} Observe that for all positive integers $n \geq 13$, we have
\[
  \sum_{k=1}^{n} \dfrac{\mu(k)}{F_k} = \sum_{k=1}^{12} \dfrac{\mu(k)}{F_k} + \sum_{k=13}^{n} \dfrac{\mu(k)}{F_k}.
\]
Following the same steps as before, we can easily verify that (\ref{kou1}) holds.
Notice that regarding the value of the constant $C$ in (\ref{Ca}), one attains high precision convergence for the series of the theorem,
by invoking relations (\ref{kou4}) and ({\ref{kou5}}) (mainly, due to the growth of the Fibonacci numbers). For improved numerical accuracy, one extends the partial sum to $k_{0}=60$, which allows us to determine the constant $C$ to the stated number of decimal places.
\end{proof}

The asymptotics of the determinant $\det (F_R (n))$ are presented in the following result.

\begin{theorem}
Let $F_R (n)$ be the $n \times n$ Fibonacci--Redheffer matrix defined by (\ref{2}).
Then, for the leading behaviour of the product of its eigenvalues as $n \rightarrow \infty$, we have
\begin{equation}
     \det (F_R (n)) \sim C_0\, \phi^{n \frac{(n+1)}{2}}5^{-\frac{n}{2}},   \label{ft}
\end{equation}
where
\begin{equation}
C_0\approx -0.7921376. \label{123}
\end{equation}
In particular, $C_0 = C \cdot C_\phi$, where $C$ is defined in (\ref{kou}) and (\ref{Ca}), and
\begin{equation}
    C_\phi := (1-b) \bigl (1-b^{2}\bigr ) \bigl (1-b^{3}\bigr ) \cdots \approx 1.226742
    \quad \mbox{with $b := - \phi^{-2}$.}       \label{b}
\end{equation}
\end{theorem}

\begin{proof}
Let $b$ as defined in (\ref{b}).
It is an easy exercise for one to have an asymptotic expression for the Fibonacci factorial.
In particular,
\begin{eqnarray}
 \sum_{k=1}^{n}\ln F_k &=&  \sum_{k=1}^{n}\bigl ( \ln \phi ^k-\ln \sqrt{5}+\ln \bigl (1-b^k \bigr ) \bigr ) \nonumber \\
                       &=&  \dfrac{n(n+1)}{2}\ln \phi -\dfrac{n}{2}\ln5+\sum_{k=1}^{\infty}\ln\bigl (1-b^k \bigr )-\sum_{k=n+1}^{\infty}\ln \bigl (1-b^k \bigr ) .  \label{rial}
\end{eqnarray}
As $n\rightarrow \infty$ the last sum of (\ref{rial}) is $\mathcal{O}\left(\left|b^{n+1}\right|\right)$. Hence,
\[
     \prod_{k=1}^{n}F_k \sim C_\phi\, \phi^{n \frac{(n+1)}{2}} 5^{-\frac{n}{2}} ,
\]
see, e.g., \cite{KPG} for details. Notice that we have used the fact that the infinite product appearing in (\ref{b}), clearly, converges. Now, (\ref{ft}) follows immediately by invoking (\ref{kou}), (\ref{Ca}) and (\ref{rial}) in Theorem \ref{firsttheorem}, with $C_0$ as defined in (\ref{123}). $\hfill$
\end{proof}

%%%%%%%%%%%%%%%%%%%%%%%%%%%%%%%%%%%%%%%%%%%%%%%%%%%%%%%%%%%%%%%%%%%%%%%%%%%%%%%%%%%%%%%%%%%%%%%%%%%%%%%%%%%%%
\section{Spectral Properties of the Fibonacci--Redheffer Matrix}  \label{section3}

For clarity, let us first present a short overview regarding well known results for the original Redheffer matrix defined in Section \ref{section1}.
The matrix $R_{n}$ has a simple real positive eigenvalue $\lambda^{+}$ asymptotic to $\sqrt{n}$, and a simple real negative eigenvalue $\lambda^{-}$ asymptotic to $-\sqrt{n}$, as well.
In particular,
\[
      \lambda^{\pm} = {\pm} \sqrt{n} + \log \sqrt{n} + \gamma - \dfrac{3}{2} + O \biggl (\dfrac{\log^{2} n}{\sqrt{n}} \biggr )  \quad \text{as $n \rightarrow \infty$,}
\]
where $\gamma = 0.5772\cdots$ is the Euler$-$Mascheroni constant.
Furthermore, $\lambda = 1$ is an eigenvalue of algebraic multiplicity $n - \lfloor \log_{2}n\rfloor -1$.           %%% (where, $\lfloor n\rfloor$ denotes, as usal, the largest integer that doesn't exceed $n$).
Finally, there are $\lfloor \log_{2}n\rfloor -1$ ``small'' eigenvalues meaning that given an $\varepsilon > 0$,
and for sufficiently large $n$, these eigenvalues have relatively small moduli (see \cite{BJ}).
In particular, they lie inside the circle $\bigl \{ z \in \mathbb{C} :\, |z| < \log_{2-\varepsilon} n \bigr \}$.
Hence, for the spectral radius $\rho(R_{n})$ of $R_{n}$, we have
\[
          \rho(R_{n})\sim \sqrt{n}  \quad \text{as $n \rightarrow \infty$.}
\]
A remarkable conjecture states that the ``small'' eigenvalues have negative real part; see \cite{BJ}.

Regarding the eigenvectors, it is known that the eigenspace of $R_{n}$ corresponding to the eigenvalue $1$ has dimension $\bigl \lceil \frac{n}{2} \bigr \rceil - 1$,
where $\lceil x \rceil$ is the {\it ceiling function}, that is, the upper integer part of $x$.
As a consequence, $R_{n}$ is not diagonalizable for $n\geq 5$.
Every eigenspace of $R_{n}$ corresponding to an eigenvalue different than $1$ has dimension $1$.
Finally, Robinson and Barrett \cite{RB} studied systematically the Jordan structure of the matrix corresponding to the eigenvalue $1$.

We return now to the Fibonacci--Redheffer matrix $F_{R}(n)$ defined by (\ref{2}).
As we will see, the spectral structures of $F_{R}(n)$ and $R_n$ appear to have many differences, although the two matrices are entrywise nonnegative and have the same zero pattern.
These differences are due to the fact that the Fibonacci sequence is rapidly increasing.

Applying Theorems \ref{firsttheorem} and \ref{secondtheorem}, one immediately verifies that the Fibonacci--Redheffer matrix is always nonsingular for $n\geq 3$
(also true for the trivial case $n=1$, while for $n=2$, the matrix has zero determinant).
Observe that, this is not always true for the classical Redheffer matrix $R_{n}$;
recall that its determinant is equal to the Mertens function (see relation (\ref{1})), which has an infinite number of zeros.
To proceed, we need the following lemmas.

\begin{lemma}  \label{firstnonzero}
Let $\lambda$ be an eigenvalue of the Fibonacci--Redheffer matrix $F_{R}(n)$, with a corresponding eigenvector $\mathbf{x} = [ x_i ]$.
Then, the first entry $x_1$ of the eigenvector $\mathbf{x}$ is nonzero.
\end{lemma}

\begin{proof}
For the eigenvalue $\lambda$ and the corresponding eigenvector $\mathbf{x} = [ x_i ]$ of $F_{R}(n)$, consider the equation
\begin{equation}  \label{LS}
   ( F_{R}(n) - \lambda \, I_n ) \mathbf{x} = \mathbf{0}  ,
\end{equation}
that is,
\[
   \begin{bmatrix}
   1-\lambda & 1 & 1 & 1 & 1 & 1  & 1  & \cdots & 1 \\
    1 & 1-\lambda & 0 & 1 & 0 & 1 & 0  & \cdots & * \\
    1 & 0 & 2-\lambda & 0 & 0 & 2 & 0  & \cdots & * \\
    1 & 0 & 0 & 3-\lambda & 0 & 0 & 0  & \cdots & * \\
    1 & 0 & 0 & 0 & 5-\lambda & 0 & 0  & \cdots & * \\
    1 & 0 & 0 & 0 & 0 & 8-\lambda & 0  & \cdots & * \\
    1 & 0 & 0 & 0 & 0 & 0 & 13-\lambda & \cdots & * \\
    \vdots & \vdots & \vdots & \vdots & \vdots & \vdots & \vdots & \ddots & \vdots \\
    1 & 0 & 0 & 0 & 0 & 0 & 0 & \cdots & F_n - \lambda
   \end{bmatrix}
   \begin{bmatrix}
   x_1 \\ x_2 \\ x_3 \\ x_4 \\ x_5 \\ x_6 \\ x_7 \\ \vdots \\ x_n
   \end{bmatrix}
   =
   \begin{bmatrix}
   0 \\ 0 \\ 0 \\ 0 \\ 0 \\ 0 \\ 0 \\ \vdots \\ 0
   \end{bmatrix}  .
\]
For the sake of contradiction, assume that $x_1 = 0$. Then we get
\begin{equation}  \label{cut1}
   \begin{bmatrix}
    1 & 1 & 1 & 1 & 1  & 1  & \cdots & 1 \\
    1-\lambda & 0 & 1 & 0 & 1 & 0  & \cdots & * \\
    0 & 2-\lambda & 0 & 0 & 2 & 0  & \cdots & * \\
    0 & 0 & 3-\lambda & 0 & 0 & 0  & \cdots & * \\
    0 & 0 & 0 & 5-\lambda & 0 & 0  & \cdots & * \\
    0 & 0 & 0 & 0 & 8-\lambda & 0  & \cdots & * \\
    0 & 0 & 0 & 0 & 0 & 13-\lambda & \cdots & * \\
    \vdots  & \vdots & \vdots & \vdots & \vdots & \vdots & & \vdots \\
    0 & 0 & 0 & 0 & 0 & 0 & \cdots & F_n - \lambda
   \end{bmatrix}
   \begin{bmatrix}
   x_2 \\ x_3 \\ x_4 \\ x_5 \\ x_6 \\ x_7 \\ \vdots \\ x_n
   \end{bmatrix}
   =
   \begin{bmatrix}
   0 \\ 0 \\ 0 \\ 0 \\ 0 \\ 0 \\ \vdots \\ 0
   \end{bmatrix} .
\end{equation}

We consider the following two cases.

$(a)\;$  Suppose that the eigenvalue $\lambda$ is different than the Fibonacci numbers $F_2 , F_3 , \ldots , F_n$.
Then it is apparent that the rank of the $n \times (n-1)$ coefficient matrix of (\ref{cut1}) is equal to $n-1$.
As a consequence, $x_2 = x_3 = \cdots = x_n = 0 = x_1$, which is a contradiction.

$(b)\;$  Suppose that the eigenvalue $\lambda$ is equal to one of the Fibonacci numbers $F_2 , F_3 , \ldots , F_n$, say $F_k$.
If $\lambda = 1$, then the rank of the $n \times (n-1)$ coefficient matrix of (\ref{cut1}) is equal to $n-1$, and hence, $x_2 = x_3 = \cdots = x_n = 0 = x_1$, which is a contradiction.
If $\lambda > 1$, then by applying appropriate row operations, the first row of the $n \times (n-1)$ coefficient matrix of (\ref{cut1}) can be transformed into the form
\[
     \begin{bmatrix}
     \;\; 0 & 0 & \cdots & 0 & 1 - \hspace{-2mm} \sum\limits_{\tiny{\begin{array}{c} i\mid k \\ i\ne 1,k   \end{array}}} \hspace{-2mm} \dfrac{F_i}{F_i - F_k} & * & * & \cdots & * \;\;
     \end{bmatrix} ,
\]
where
\[
     1 - \hspace{-2mm} \sum\limits_{\tiny{\begin{array}{c} i\mid k \\ i\ne 1,k \end{array}}} \hspace{-2mm} \dfrac{F_i}{F_i - F_k} > 1
\]
is the $(k-1)$-th entry of the row.
Again, we get that the rank of the $n \times (n-1)$ coefficient matrix of (\ref{cut1}) is equal to $n-1$.
As a consequence, $x_2 = x_3 = \cdots = x_n = 0 = x_1$, which is a contradiction.

Finally, we conclude that $x_1 \ne 0$.              $\hfill$
\end{proof}

\begin{lemma}  \label{geom}
The geometric multiplicity of any eigenvalue of the Fibonacci--Redheffer matrix $F_{R}(n)$ is equal to $1$.
\end{lemma}

\begin{proof}
If we assume that the matrix $F_{R}(n)$ has two linearly independent eigenvectors corresponding to the same eigenvalue,
then some linear combination of them will have its first entry equal to $0$; this is a contradiction.       $\hfill$
\end{proof}

Stimulating communication with Professor Sergey Savchenko, on rank one perturbations, led to the following lemma;
its proof is a simple exercise presented here for clarity.

\begin{lemma}  \label{sav}
Consider an $n\times n$ matrix $B$ and the $n\times n$ rank one perturbation matrix
\[
  M(n) = \mathbf{u} \cdot \mathbf{v}^{T}  \,( = \mathbf{u} \, \mathbf{v}^{*} )\,  =
  \begin{bmatrix}
  1 \\
  1 \\
  1 \\
  \vdots \\
  1
  \end{bmatrix}
  \begin{bmatrix}
  1 & 0 & 0 & \cdots & 0
  \end{bmatrix}
  =
  \begin{bmatrix}
  1 & 0 & 0 & \cdots & 0 \\
  1 & 0 & 0 & \cdots & 0 \\
  1 & 0 & 0 & \cdots & 0 \\
  \vdots & \vdots & \vdots & \ddots & \vdots \\
  1 & 0 & 0 & \cdots & 0
  \end{bmatrix} .
\]
Let $\lambda$ be an eigenvalue of $B$ of geometric multiplicity $1$, with corresponding left eigenvector $\mathbf{y}$,
and let $\hat{\lambda}$ be an eigenvalue of $B \pm M$ of geometric multiplicity $1$, with corresponding eigenvector $\hat{\mathbf{x}}$.
If $\mathbf{y}^* \mathbf{u} \ne 0$ and $\mathbf{v}^* \hat{\mathbf{x}} \ne 0$, then $\lambda \ne \hat{\lambda}$.
\end{lemma}

\begin{proof}
Since the geometric multiplicities of $\lambda$ and $\hat{\lambda}$ are equal to $1$, the (nonzero) vectors $\mathbf{y}$ and $\hat{\mathbf{x}}$ are unique up to scalar multiple.
They also satisfy $\mathbf{y}^* B = \lambda\, \mathbf{y}^*$ and $( B \pm \mathbf{u} \, \mathbf{v}^*) \hat{\mathbf{x}} = \hat{\lambda}\, \hat{\mathbf{x}}$.

Suppose that $\mathbf{y}^* \mathbf{u} \ne 0$ and $\mathbf{v}^* \hat{\mathbf{x}} \ne 0$.
Taking the inner product of both sides of the equation
\[
         B\, \hat{\mathbf{x}} = \hat{\lambda}\, \hat{\mathbf{x}} \mp \mathbf{u} \, \mathbf{v}^* \hat{\mathbf{x}}
\]
with the left eigenvector $\mathbf{y}$ yields
\[
     \mathbf{y}^* B\, \hat{\mathbf{x}} = \mathbf{y}^* \hat{\lambda}\, \hat{\mathbf{x}} \mp \mathbf{y}^* \mathbf{u} \, \mathbf{v}^* \hat{\mathbf{x}} ,
\]
or equivalently,
\[
     \lambda ( \mathbf{y}^* \hat{\mathbf{x}} ) = \hat{\lambda} ( \mathbf{y}^*  \hat{\mathbf{x}} ) \mp ( \mathbf{y}^* \mathbf{u} ) ( \mathbf{v}^* \hat{\mathbf{x}} ) .
\]
If $\mathbf{y}^*\hat{\mathbf{x}} \ne 0$, then it is apparent that
\[
     \lambda = \hat{\lambda} \mp \frac{  ( \mathbf{y}^* \mathbf{u} ) ( \mathbf{v}^* \hat{\mathbf{x}} ) }{\mathbf{y}^*\hat{\mathbf{x}}} \ne \hat{\lambda} .
\]
If $\mathbf{y}^*\hat{\mathbf{x}} = 0$, then it is apparent that $\mathbf{y}^* \mathbf{u} = 0$ or $\mathbf{v}^* \hat{\mathbf{x}} = 0$, which is a contradiction.      $\hfill$
\end{proof}

For the $n \times n$ Fibonacci--Redheffer matrix $F_R (n)$, consider
the $n \times n$ singular upper triangular matrix
\[
  T(n) =
  \begin{bmatrix}
  0 & 1 & 1 & 1 & 1 & 1 & \cdots & 1 \\
  0 & 1 & 0 & 1 & 0 & 1 & \cdots & * \\
  0 & 0 & 2 & 0 & 0 & 2 & \cdots & * \\
  0 & 0 & 0 & 3 & 0 & 0 & \cdots & * \\
  0 & 0 & 0 & 0 & 5 & 0 & \cdots & * \\
  0 & 0 & 0 & 0 & 0 & 8 & \cdots & * \\
  \vdots & \vdots & \vdots & \vdots & \vdots & \vdots & \ddots & \vdots \\
  0 & 0 & 0 & 0 & 0 & 0 & \cdots & F_n
  \end{bmatrix}
\]
and the $n \times n$ rank one matrix $M(n)$ in Lemma \ref{sav}.
Of course, all the eigenvalues of $T(n)$ are simple, the nonzero eigenvalues of $T(n)$ are the Fibonacci numbers $F_2 , F_3 , \ldots , F_n$, and
$F_R (n) = T (n)  +  M (n)$ (or equivalently, $T (n) =  F_R (n) -  M (n)$). Moreover, we observe the following.

\begin{remark}  \label{TT} \rm
Consider the vectors $\mathbf{u}$ and $\mathbf{v}$ and the $n \times n$ matrix $T(n) = F_R (n) - \mathbf{u} \cdot \mathbf{v}^{T}$ above.
Following exactly the arguments of the proof of Lemma \ref{firstnonzero}, one can obtain that for any eigenvector $\hat{\mathbf{x}}$ of $T(n)$,
its first entry is nonzero, or equivalently, $\mathbf{v}^* \hat{\mathbf{x}} \ne 0$.
\end{remark}

If $\mathbf{y} = [ y_i ] = \begin{bmatrix} y_1 & y_2 & \cdots & y_n \end{bmatrix}^T$ is a left eigenvector of the matrix $F_R (n)$ corresponding to a real eigenvalue $\lambda$ of $F_R (n)$,
then it is easy to see that
\begin{equation}
       \sum_{i=1}^{n} y_{i} = \lambda\, y_{1} .        \label{1CONT}
\end{equation}
Since $0$ cannot be an eigenvalue of $F_R (n)$, (\ref{1CONT}) yields
\begin{equation}
    \sum_{i=1}^{n} y_{i} \neq 0 \quad \text{if and only if} \quad y_{1} \neq 0.      \label{CONT}
\end{equation}
We are now ready to apply Lemmas \ref{firstnonzero} and \ref{geom} and Remark \ref{TT}, in order to prove the following.

\begin{theorem}  \label{different}
Let $\sigma (F_R(n))$ be the spectrum of the matrix $F_R (n)$.
Then, for any positive integer $n \geq 2$, we have
\[
       F_{j} \notin \sigma (F_R(n)) ,  \quad j=1,2,\ldots, n.
\]
In words, none of the first $n$ Fibonacci numbers is an eigenvalue of $F_R (n)$.
\end{theorem}

\begin{proof}
Consider a real eigenvalue $\lambda \in \sigma (F_R(n))$ and a corresponding left eigenvector $\mathbf{y} = [ y_i ]$.
By Lemma \ref{geom}, the geometric multiplicity of the eigenvalue $\lambda$ is equal to $1$.
Hence, by Lemma \ref{sav}, Remark \ref{TT}, relation (\ref{CONT}), and the above discussion, it suffices to prove that the first entry $y_1$ of the left eigenvector $\mathbf{y}$ is nonzero.

Consider the equation
\[
     \bigl (F_{R}(n)^T - \lambda\, I_n \bigr ) \mathbf{y} = \mathbf{0} ,
\]
that is,
\[
     \begin{bmatrix}
     1-\lambda & 1 & 1 & 1 & 1 & 1 & 1  & \cdots & 1 \\
     1 & 1-\lambda & 0 & 0 & 0 & 0 & 0  & \cdots & 0 \\
     1 & 0 & 2-\lambda & 0 & 0 & 0 & 0  & \cdots & 0 \\
     1 & 1 & 0 & 3-\lambda & 0 & 0 & 0  & \cdots & 0 \\
     1 & 0 & 0 & 0 & 5-\lambda & 0 & 0  & \cdots & 0 \\
     1 & 1 & 2 & 0 & 0 & 8-\lambda & 0  & \cdots & 0 \\
     1 & 0 & 0 & 0 & 0 & 0 & 13-\lambda & \cdots & 0 \\
     \vdots & \vdots & \vdots & \vdots & \vdots & \vdots & \vdots & \ddots & \vdots \\
     1 & * & * & * & * & * & * & \cdots & F_n - \lambda
     \end{bmatrix}
     \begin{bmatrix}
     y_1 \\ y_2 \\ y_3 \\ y_4 \\ y_5 \\ y_6 \\ y_7 \\ \vdots \\ y_n
     \end{bmatrix}
     =
     \begin{bmatrix}
     0 \\ 0 \\ 0 \\ 0 \\ 0 \\ 0 \\ 0 \\ \vdots \\ 0
     \end{bmatrix} .
\]
For the sake of contradiction, assume that $y_1 = 0$. Then, we get
\begin{equation}  \label{cut3}
     \begin{bmatrix}
     1 & 1 & 1 & 1 & 1 & 1  & \cdots & 1 \\
     1-\lambda & 0 & 0 & 0 & 0 & 0  & \cdots & 0 \\
     0 & 2-\lambda & 0 & 0 & 0 & 0  & \cdots & 0 \\
     1 & 0 & 3-\lambda & 0 & 0 & 0  & \cdots & 0 \\
     0 & 0 & 0 & 5-\lambda & 0 & 0  & \cdots & 0 \\
     1 & 2 & 0 & 0 & 8-\lambda & 0  & \cdots & 0 \\
     0 & 0 & 0 & 0 & 0 & 13-\lambda & \cdots & 0 \\
     \vdots & \vdots & \vdots & \vdots & \vdots & \vdots &  & \vdots \\
     * & * & * & * & * & * & \cdots & F_n - \lambda
     \end{bmatrix}
     \begin{bmatrix}
     y_2 \\ y_3 \\ y_4 \\ y_5 \\ y_6 \\ y_7 \\ \vdots \\ y_n
     \end{bmatrix}
     =
     \begin{bmatrix}
     0 \\ 0 \\ 0 \\ 0 \\ 0 \\ 0 \\ \vdots \\ 0
     \end{bmatrix}  .
\end{equation}

We consider the following two cases.

$(a)\;$  Suppose that the eigenvalue $\lambda$ is different than the Fibonacci numbers $F_2 , F_3 , \ldots , F_n$.
Then it is apparent that the rank of the $n \times (n-1)$ coefficient matrix of (\ref{cut3}) is equal to $n-1$.
As a consequence, $y_2 = y_3 = \cdots = y_n = 0 = y_1$, which is a contradiction.

$(b)\;$  Suppose that the eigenvalue $\lambda$ is equal to one of the Fibonacci numbers $F_2 , F_3 , \ldots , F_n$, say $F_k$.
If $\lambda = F_n$, then the rank of the $n \times (n-1)$ coefficient matrix of (\ref{cut3}) is equal to $n-1$, and hence, $y_2 = y_3 = \cdots = y_n = 0 = y_1$, which is a contradiction.
If $\lambda < F_n$, then by applying appropriate row operations, the first row of the $n \times (n-1)$ coefficient matrix of (\ref{cut3}) can be transformed into the form
\[
     \begin{bmatrix}
     \;\; * & * & \cdots & * & 1 - \hspace{-2mm} \sum\limits_{\tiny{\begin{array}{c} k\mid i \\ i\ne k   \end{array}}} \hspace{-2mm} \dfrac{F_k}{F_i - F_k} & 0 & 0 & \cdots & 0 \;\;
     \end{bmatrix} ,
\]
where
\begin{eqnarray*}
           1 -  \sum\limits_{\tiny{\begin{array}{c} k\mid i \\ i\ne k \end{array}}} \dfrac{F_k}{F_i - F_k}
      &=&  1 - \biggl ( \dfrac{F_{k}}{F_{2k} - F_k} + \dfrac{F_{k}}{F_{3k} - F_k} + \dfrac{F_{k}}{F_{4k} - F_k} + \cdots \text{ ``finite number of terms''} \biggr )   \\
      &>&  1 - \biggl ( \dfrac{1}{2} + \dfrac{1}{4} + \dfrac{1}{8} + \cdots \text{ ``finite number of terms''}  \biggr ) > 0
\end{eqnarray*}
is the $(k-1)$-th entry of the row.
Again, we get that the rank of the $n \times (n-1)$ coefficient matrix of (\ref{cut3}) is equal to $n-1$.
As a consequence, $y_2 = y_3 = \cdots = y_n = 0 = y_1$, which is a contradiction.

Finally, we conclude that $y_1 \ne 0$.              $\hfill$
\end{proof}

\begin{corollary} \label{nonzeroentries}
Let $\lambda$ be an eigenvalue of the Fibonacci--Redheffer matrix $F_{R}(n)$, with a corresponding eigenvector $\mathbf{x} = [ x_i ]$.
Then, all entries $x_1 , x_2 , \ldots , x_n$ of the eigenvector $\mathbf{x}$ are nonzero.
\end{corollary}

\begin{proof}
By Lemma \ref{firstnonzero}, $x_1 \ne 0$.
Consider the equation (\ref{LS}).
Subtracting the $n$-th row of the $n \times n$ coefficient matrix $F_{R}(n) - \lambda\, I_n$ from each of the rest $n-1$ rows of the matrix and moving it to the top yields
\[
   \begin{bmatrix}
    1 & 0 & 0 & 0 & 0 & 0 & 0 & \cdots & 0 & F_n - \lambda \\
    0 & 1 & 1 & 1 & 1 & 1  & 1  & \cdots & 1 & 1 - (F_n - \lambda) \\
    0 & 1-\lambda & 0 & 1 & 0 & 1 & 0  & \cdots & * & * \\
    0 & 0 & 2-\lambda & 0 & 0 & 2 & 0  & \cdots & * & * \\
    0 & 0 & 0 & 3-\lambda & 0 & 0 & 0  & \cdots & * & * \\
    0 & 0 & 0 & 0 & 5-\lambda & 0 & 0  & \cdots & * & * \\
    0 & 0 & 0 & 0 & 0 & 8-\lambda & 0  & \cdots & * & * \\
    0 & 0 & 0 & 0 & 0 & 0 & 13-\lambda & \cdots & * & * \\
    \vdots & \vdots & \vdots & \vdots & \vdots & \vdots & \vdots & \ddots & \vdots & \vdots \\
    0 & 0 & 0 & 0 & 0 & 0 & 0 & \cdots & F_{n-1} - \lambda & - (F_n - \lambda)  .
   \end{bmatrix}
\]
By Theorem \ref{different}, the entries $1-\lambda , 2-\lambda, \ldots , F_{n-1} - \lambda , F_n - \lambda$ are nonzero.
For the sake of contradiction, assume that $x_k = 0$ for some $k \in \{ 2,3,\ldots,n \}$ and remove the corresponding column from the above matrix, resulting to an $n \times (n-1)$ matrix $B_k$.
If $k \in \{ n-1 , n \}$, then it is obvious that the rank of $B_k$ is $n-1$, and hence, $x_1 = x_2 = \cdots = x_n =0$, which is a contradiction.
If $k \in \{ 2,3,\ldots,n-2 \}$, then by applying appropriate column operations and following arguments similar to those of the proofs of Lemma \ref{firstnonzero} and Theorem \ref{different},
the last column of $B_k$ can be transformed to a column with its $k$-th entry nonzero and its $(k+1)$-th, $(k+2)$-th, $\ldots$, $(n-1)$-th and $n$-th entries equal to $0$.
As a consequence, we get again that the rank of the $n \times (n-1)$ matrix $B_k$ is $n-1$, and thus, $x_1 = x_2 = \cdots = x_n =0$, which is a contradiction.        $\hfill$
\end{proof}

The main pillars of our analysis, until now, are the following:
\begin{description}
 \item[(i)  ]  The determinant of $F_{R}(n)$ is negative, and thus, $F_{R}(n)$ is nonsingular and its eigenvalues are nonzero (Theorems \ref{firsttheorem} and \ref{secondtheorem}).
 \item[(ii) ]  All eigenvalues of $F_{R}(n)$ are different than the first $n$ Fibonacci numbers (Theorem \ref{different}).
 \item[(iii)]  All eigenvalues of $F_{R}(n)$ have geometric multiplicity $1$ (Lemma \ref{geom}).
 \item[(iv) ]  For any eigenvector of $F_{R}(n)$, all its entries are nonzero (Lemma \ref{firstnonzero} and Corollary \ref{nonzeroentries}).
\end{description}
Next, we will obtain that all eigenvalues of $F_{R}(n)$ are real and simple (i.e., of algebraic multiplicity $1$).
We start by sketching a second proof of Corollary \ref{nonzeroentries}.

Recall the equation (\ref{LS}), that is,
\[
    ( F_{R}(n) - \lambda\, I_n ) \mathbf{x} = \mathbf{0} ,
\]
where $\mathbf{x} = [ x_i ] = \begin{bmatrix} x_1 & x_2 & \cdots & x_n \end{bmatrix}^T$ is an eigenvector of $F_{R}(n)$ corresponding to an eigenvalue $\lambda \in \sigma (F_{R}(n))$.
Consider the $n$ equations of this system and set
\begin{equation}
   \omega: = \bigl \lfloor \frac{n}{2} \bigr \rfloor =
   \begin{cases}
   \dfrac{n}{2},   &  \text{if $n$ is even;} \\[2mm]
   \dfrac{n-1}{2}, &  \text{if $n$ is odd.} \label{omega}
   \end{cases}
\end{equation}
>From the $(\omega+1)$-th, $(\omega+2)$-th, $\ldots$, $(n-1)$-th and $n$-th equations, we can see that
\begin{equation}
     x_{k} = \frac{1}{\lambda-F_{k}}  \,  x_1  ,     \quad k = \omega+1, \omega+2, \ldots, n  ,     \label{xk}
\end{equation}
for any eigenvector $\mathbf{x} = [ x_i ]$, verifying that the last $n-\omega$ entries of $\mathbf{x}$ are nonzero, as well as the first entry.
For a concrete instance, and motivated by Remark \ref{firstremark} where $n=8$ and $\omega=4$, we have
\[
( F_{R}(8) - \lambda I_8 ) \mathbf{x} =
\left[
\begin{array}{cccc:cccc}
1-\lambda & 1 & 1 & 1 & 1 & 1 & 1 & 1 \\
1 & 1-\lambda & 0 & 1 & 0 & 1 & 0 & 1 \\
1 & 0 & 2-\lambda & 0 & 0 & 2 & 0 & 0 \\
1 & 0 & 0 & 3-\lambda & 0 & 0 & 0 & 3 \\ \hdashline
1 & 0 & 0 & 0 & 5-\lambda & 0 & 0 & 0 \\
1 & 0 & 0 & 0 & 0 & 8-\lambda & 0 & 0 \\
1 & 0 & 0 & 0 & 0 & 0 & 13-\lambda & 0 \\
1 & 0 & 0 & 0 & 0 & 0 & 0 & 21-\lambda
\end{array}
\right]
\begin{bmatrix}
x_1 \\ x_2 \\ x_3 \\ x_4 \\
x_5 \\ x_6 \\ x_7 \\ x_8
\end{bmatrix}
=
\begin{bmatrix}
0 \\ 0 \\ 0 \\ 0 \\
0 \\ 0 \\ 0 \\ 0
\end{bmatrix}.
\]
and the last four equations imply
$x_{5} = \dfrac{1}{\lambda - F_{5}} \, x_1$, $x_{6} = \dfrac{1}{\lambda - F_{6}} \, x_1$, $x_{7} = \dfrac{1}{\lambda - F_{7}} \, x_1$, and $x_{8} = \dfrac{1}{\lambda - F_{8}} \, x_1$.
Let us now focus on the first $\omega$ equations of the system (\ref{LS}).
We observe that the number of terms of the $j$-th equation ($j \le \omega$) is
\begin{equation*}
  \begin{cases}
  \bigl \lfloor \dfrac{n}{j} \bigr \rfloor + 1,& \text{if $j\neq 1$;} \\[2mm]
                                       n,      & \text{if $j=1$.}
  \end{cases}
\end{equation*}
Starting from the $\omega$-th equation, which always has $3$ terms (due to (\ref{omega})), we derive
\[
    x_1 + (F_{\omega} - \lambda) x_{\omega} + F_{\omega} x_{2\omega} = 0 \;\Longleftrightarrow\;
    x_{\omega} = \dfrac{1}{\lambda-F_{\omega}} \biggl ( 1 + \dfrac{F_{\omega}}{\lambda-F_{2\omega}} \biggr ) x_1  ,
\]
where we have used (\ref{xk}). Notice that, since $\det(F_{R}(n))<0$ and due to (\ref{xk}), we confirm that the entry $x_{\omega}$ is also a nonzero multiple of $x_1$.
Working backwards with the $(\omega-1)$-th, $(\omega-2)$-th, $\ldots$, $3$-rd and $2$-nd equations, and in a similar way, we obtain that
the entries $x_{\omega-1}, x_{\omega-2},\ldots, x_3, x_2$ of the corresponding eigenvector $\mathbf{x}$ are also nonzero multiples of $x_1$.
Thus, all entries of all eigenvectors of $F_R (n)$ are nonzero. From the first equation of the system (\ref{LS}), that is,
$( 1 - \lambda ) x_1 + x_2 + \cdots + x_{\omega} + x_{\omega+1} + \cdots + x_n = 0$,
we get
\begin{eqnarray}
  0 &=& ( 1 - \lambda )x_1 + \frac{1}{\lambda -F_{2}} ( x_{1} + x_{4} + x_{6} + x_{8} + \cdots ) +\frac{1}{\lambda -F_{3}} ( x_{1} + F_{3} x_{6} + F_{3} x_{9} + \cdots )  \nonumber \\
    & & +\, \frac{1}{\lambda -F_{4}} ( x_{1} + F_{4} x_{8} + \cdots ) + \cdots + \frac{1}{\lambda -F_{\omega}} ( x_{1} + F_{\omega} x_{2\omega} + \cdots ) + \sum_{k=\omega +1}^{n} \frac{1}{\lambda-F_{k}}x_1  .
        \;\;\;\;\;\;\;\;\;  \label{xg}
\end{eqnarray}
The right-hand side of (\ref{xg}) can be written in the form $Q_{F_{R}(n)} (\lambda) \cdot x_1$, where $Q_{F_{R}(n)} (z)$ is a continuous function in $z \in \mathbb{R} \setminus \{F_2 , F_3 , \ldots , F_n \}$.
For example, in the case of the matrix of Remark \ref{firstremark} (where $n=8$), we have
\begin{eqnarray}
   Q_{F_{R}(8)} (\lambda) &=& (1-\lambda) + \underbrace{\frac{1}{\lambda-1} \biggl [ 1 + 1\underbrace{\cdot \frac{1}{\lambda-3} \biggl ( 1 + 3 \cdot
                        \frac{1}{\lambda-21} \biggr )}_{x_{4}/x_{1}}+\underbrace{\frac{1}{\lambda-8}}_{x_6/x_1}+\underbrace{\frac{1}{\lambda-21}}_{x_8/x_1} \biggr ]}_{x_{2}/x_{1}}   \nonumber    \\
                    & & + \underbrace{\frac{1}{\lambda-2} \biggl ( 1+2\cdot \frac{1}{\lambda-8} \biggr )}_{x_3/x_1}
                        + \underbrace{ \frac{1}{\lambda-3} \biggl (1+3\cdot \frac{1}{\lambda-21} \biggr )}_{x_4/x_1}+\sum_{k=5}^{8}\underbrace{\frac{1}{\lambda-F_{k}}}_{x_{k}/x_{1}} \label{ex01} \\
                    &=& 0 .  \nonumber
\end{eqnarray}
By the above discussion, we get that $\lambda$ is an eigenvalue of $F_{R}(n)$ if and only if (\ref{LS}) holds for some nonzero $\mathbf{x}$, or equivalently, if and only if $Q_{F_{R}(n)} (\lambda) = 0$.
Moreover, $-(z-1) (z-2) \cdots (z-F_n) Q_{F_{R}(n)} (\lambda)$ is a monic polynomial of degree $n$.
The following lemma is necessary in locating the eigenvalues of $F_{R}(n)$ on the real axis.

\begin{lemma}  \label{limits}
For the function $Q_{F_{R}(n)}(z)$ $(z \in \mathbb{R} \setminus \{F_2 , F_3 , \ldots , F_n \})$ introduced above, it holds
\[
     \lim_{z \rightarrow -\infty}Q_{F_{R}(n)} (z) = + \infty
     \quad \mbox{and} \quad
     \lim_{z \rightarrow +\infty}Q_{F_{R}(n)} (z) = - \infty ,
\]
and for any $k = 2 , 3 , \dots , n$,
\[
     \lim_{z \rightarrow F_{k}^{-}}Q_{F_{R}(n)} (z) = - \infty
     \quad \mbox{and} \quad
     \lim_{z \rightarrow F_{k}^{+}}Q_{F_{R}(n)} (z) = +\infty.
\]
\end{lemma}

\begin{proof}
Observe that $Q_{F_{R}(n)}(z)$ can be written in the form
\begin{equation}   \label{QQQ}
Q_{F_{R}(n)}(z)
=
1-z+
\sum_{i=2}^{n}
\frac{1}{z-F_i}
\left(
1+F_i
\sum_{m=2}^{\lfloor n/i\rfloor}
\frac{x_{mi}}{x_1}
\right),
\end{equation}
where
\begin{equation}  \label{rat}
\frac{x_i}{x_1}
=
\begin{cases}
\frac{1}{z-F_i} \left( 1+F_i \sum\limits_{m=2}^{\lfloor n/i\rfloor} \frac{x_{mi}}{x_1} \right), & \text{if $i \le \dfrac{n}{2}$;} \\
          \frac{1}{z-F_i},                                                               & \text{otherwise.}
\end{cases}
\end{equation}
Moreover, it is an easy exercise to verify that, for $i,j\geq 2$, $F_{2i}\geq 3 F_i$ and $F_{j+2}\geq 2 F_j$,
which imply $F_{mi}\geq 3 \cdot 2^{m-2}F_i$, $m\geq 2$.
Thus, for any $i\geq 2$,
\begin{equation}
\begin{aligned}
\sum_{m=2}^{\infty}
\frac{F_i}{F_{mi}-F_i}
&<1.
\end{aligned}
\label{1ineq}
\end{equation}

Now, let us consider a fixed $k\in\{2,3,\ldots,n\}$.
Using (\ref{1ineq}) and applying induction, we can see that, for every $i>k$, the ratio $\dfrac{x_i}{x_1}$ is analytic at $F_k$ and
\begin{equation}
0<
1+F_i\sum_{m=2}^{\lfloor n/i\rfloor}
\left.\frac{x_{mi}}{x_1}\right|_{z=F_k}
\leq 1.
\label{analytic}
\end{equation}
In particular, for $i=n>k$, the assertion follows immediately from (\ref{rat}) since the sum is empty. Suppose that (\ref{analytic}) holds for all indices greater than $i>k$.
Since $mi>i$, the induction hypothesis and (\ref{rat}) yield
\[
-\frac{1}{F_{mi}-F_k}
\leq
\left.\frac{x_{mi}}{x_1}\right|_{z=F_k}
<0.
\]
Consequently,
\[
1
\geq 1 + F_i\sum_{m=2}^{\lfloor n/i\rfloor} \left.\frac{x_{mi}}{x_1}\right|_{z=F_k}
\geq 1 - \sum_{m=2}^{\lfloor n/i\rfloor} \frac{F_i}{F_{mi}-F_k}
\geq 1 - \sum_{m=2}^{\lfloor n/i\rfloor} \frac{F_i}{F_{mi}-F_i}
> 0 ,
\]
where the last inequality follows from (\ref{1ineq}). Taking $i=k$ in (\ref{rat}) and using (\ref{1ineq})--(\ref{analytic}), we obtain
\begin{equation}  \label{E}
\lim_{z\to F_k} (z-F_k)\frac{x_k}{x_1}
= 1 + F_k\sum_{m=2}^{\lfloor n/k\rfloor} \left.\frac{x_{mk}}{x_1}\right|_{z=F_k}
\geq 1 - \sum_{m=2}^{\lfloor n/k\rfloor} \frac{F_k}{F_{mk}-F_k}
> 0 .
\end{equation}
For $k=n$, the sums in (\ref{E}) are empty and the limit is equal to $1$.
Thus, $\dfrac{x_k}{x_1}$ has a simple pole at $F_k$ with a positive residue.

Next, we will prove by induction that, for $i=2,3,\ldots,n$,
\begin{equation} \label{2sid}
\begin{cases}
\lim\limits_{z\to F_k}(z-F_k)\frac{x_i}{x_1} > 0, & \text{if $i\mid k$;}\\
\lim\limits_{z\to F_k}(z-F_k)\frac{x_i}{x_1} = 0, & \text{if $i\nmid k$.}
\end{cases}
\end{equation}
Notice that all these limits are finite.
For $i>k$, the limit is zero by analyticity, whereas for $i=k$, the assertion follows from (\ref{E}).
Suppose that $i<k$ and that the induction hypothesis holds for all indices greater than $i$.
>From (\ref{rat}), we get
\begin{equation}
\begin{aligned}
&\lim_{z\to F_k}
(z-F_k)\frac{x_i}{x_1}=
\frac{F_i}{F_k-F_i}
\sum_{m=2}^{\lfloor n/i\rfloor}
\lim_{z\to F_k}
(z-F_k)\frac{x_{mi}}{x_1}.\\
\end{aligned}
\label{G}
\end{equation}
If $i\nmid k$, then $mi\nmid k$ for every $m\geq 2$, so all terms in (\ref{G}) vanish.
If $i\mid k$, then the sum contains the positive term corresponding to $m = \dfrac{k}{i}$, while all remaining terms are nonnegative by the induction hypothesis.
Since $\dfrac{F_i}{(F_k-F_i)} > 0$, the limit in (\ref{G}) is positive.
This completes the proof of (\ref{2sid}).

By the definition of $Q_{F_{R}(n)}(z)$ in (\ref{QQQ}) and relation (\ref{2sid}), it follows
\begin{equation*}
0<
\lim_{z\to F_k}
(z-F_k)Q_{F_R(n)}(z)
<+\infty.
\end{equation*}
Hence, for any $k = 2 , 3 , \ldots , n$, we have $\lim\limits_{z \rightarrow F_{k}^{-}}Q_{F_{R}(n)} (z) = - \infty$ and $\lim\limits_{z \rightarrow F_{k}^{+}}Q_{F_{R}(n)} (z) = +\infty$.
\end{proof}

By Lemma \ref{limits}, the lines $\text{Re} (z) = F_{k}$ ($k = 2 ,  3, \ldots , n$),
where $\text{Re} (z)$ denotes the real part of the scalar $z\in \mathbb{C}$,
are all vertical asymptotes of the graph of $Q_{F_{R}(n)} (z)$ (see Figure 1 below).
By invoking Bolzano's theorem with regard to the eigenvalues of $F_R (n)$,
and keeping in mind that the determinant of $F_R (n)$ is negative, we obtain the following result.

\begin{theorem} \label{Bolzano}
All eigenvalues of the Fibonacci--Redheffer matrix $F_R (n)$ are real and simple.
Furthermore, if $\lambda_1 < \lambda_2 < \cdots < \lambda_n$ are the $n$ simple eigenvalues of $F_R (n)$, then
\begin{align}
  &\lambda_{1} < 0 ,                   \label{e1} \\
  &F_i <\lambda_i <F_{i+1} , \quad  i=2,3,\ldots,n-1, \label{eg} \\
  &\lambda_{n} > F_{n}   .                              \label{em}
\end{align}
\end{theorem}

\begin{remark}  \rm
By the Gershgorin circle theorem (see, e.g., \cite{HJ}, pp. 344--346), the eigenvalues of $F_R (n)$ lie in the union of the closed circular disks
\[
   \mathcal{D}_i (F_R (n)) = \biggl \{ z \in \mathbb{C} :\, | z - F_R (i,i) | \le \sum_{j \ne i} | F_R (i,j) | \biggr \} , \quad i = 1 , 2 , \ldots , n ,
\]
known as the \textit{Gershgorin disks} of $F_R (n)$.
We observe that, for $n \geq 10$, each one of the Gershgorin disks
\[
   \mathcal{D}_i (F_R (n)) = \bigl \{ z \in \mathbb{C} :\, | z - F_i | \le 1  \bigr \} , \quad i = \omega+2 , \omega+3 , \ldots , n ,
\]
has no common points with any other Gershgorin disk of $F_R (n)$ and its radius is equal to $1$; it is worth mentioning that the disk
$\mathcal{D}_{\omega+1} (F_R (n)) = \bigl \{ z \in \mathbb{C} :\, | z - F_{\omega+1} | \le 1 \bigr \}$ always lies in the disk
$\mathcal{D}_{\omega} (F_R (n)) = \bigl \{ z \in \mathbb{C} :\, | z - F_{\omega} | \le 1 + F_{\omega} \bigr \}$.
As a consequence, for $n \geq 10$,
\begin{equation} \label{bound1}
    F_i < \lambda_i < F_i + 1 , \quad i = \omega+2 , \omega+3 , \ldots , n .
\end{equation}
\end{remark}

For the negative eigenvalue $\lambda_1$ of the Fibonacci--Redheffer matrix $F_R (n)$, we will prove the following nice enclosure.

\begin{theorem} \label{L1}
For the smallest eigenvalue $\lambda_1$ of the Fibonacci--Redheffer matrix $F_R (n)$, we have
\begin{equation}
     - 1 < \lambda_{1} < 0.      \label{l1}
\end{equation}
\end{theorem}

To prove (\ref{l1}) it is enough for one to verify that $\det (F_{R}(n) + I_{n} ) > 0$.
As in Theorem \ref{firsttheorem}, we have $F_R (n) + I_{n} = C (n) + \hat{D}(n)$, where $\hat{D}(n) = D(n) + I_n$,
and $C(n)$ and $D(n)$ are as defined in (\ref{C}) and (\ref{D}), respectively.
The following lemma can be obtained by straightforward calculations (see also the proof of Lemma \ref{firstlemma}).

\begin{lemma} \label{secondlemma}
The inverse of the matrix $\hat{D}(n)$ is an upper-triangular matrix given by the recursive formula
\[
   \hat{D}(n)^{-1} := \bigl [ \hat{\delta}_{i,j} \bigr ] , \;\; \text{where} \quad
   \hat{\delta}_{i,j} =
   \begin{cases}
     \dfrac{1}{F_j+1},                                                                     & \text{if $i=j$;} \\[10pt]
     -\dfrac{F_i}{F_i+1}\displaystyle\sum_{m=2}^{\lfloor j/i\rfloor}\hat{D}^{-1}_{mi,\,j}, & \text{if $i\mid j$ and $i \ne j$;} \\[10pt]
     0,                                                                                    & \text{otherwise.}
   \end{cases}
\]
\end{lemma}

For example,
\[
  \hat{D}^{-1}(8)=
  \begin{bmatrix}
  \frac12 & -\frac14 & -\frac16 & -\frac1{16} & -\frac1{12} & \frac1{108} & -\frac1{28} & -\frac1{352}\\
  0 & \frac12 & 0 & -\frac18 & 0 & -\frac1{18} & 0 & -\frac1{176}\\
  0 & 0 & \frac13 & 0 & 0 & -\frac2{27} & 0 & 0\\
  0 & 0 & 0 & \frac14 & 0 & 0 & 0 & -\frac3{88}\\
  0 & 0 & 0 & 0 & \frac16 & 0 & 0 & 0\\
  0 & 0 & 0 & 0 & 0 & \frac19 & 0 & 0\\
  0 & 0 & 0 & 0 & 0 & 0 & \frac1{14} & 0\\
  0 & 0 & 0 & 0 & 0 & 0 & 0 & \frac1{22}
  \end{bmatrix}.
\]

\noindent \textit{Proof of Theorem \ref{L1}.}\
As in Theorem \ref{firsttheorem}, it suffices to prove that the $(1,1)$-th entry of the matrix $I_n + \hat{D}^{-1} C $ is positive.
Let $\mathbf{c} = \begin{bmatrix} 0 & 1 & 1 & \cdots & 1 \end{bmatrix}^{T}$
and consider the column vector $\mathbf{x} = \hat{D}^{-1} \mathbf{c}$.
The first row of $\hat{D} \mathbf{x} = \mathbf{c}$ yields $2 x_{1} + \sum_{j=2}^{n} x_{j} = 0$.
As a consequence, we only have to prove that $\sum\limits_{j=2}^{n}x_{j} < 2$.
It is an easy exercise (induction) to see that
\[
    0 \leq x_{j} \leq \frac{1}{F_{j}+1} , \quad j = 2 , 3 , \ldots , n .
\]
We have (keeping in mind (\ref{phi}))
\begin{align*}
  \sum_{j=2}^{n}x_{j} & \leq \sum_{j=2}^{n}\frac{1}{F_{j}+1}<\sum_{j=2}^{\infty}\frac{1}{F_{j}+1} = \sum_{j=2}^{8}\frac{1}{F_{j}+1}+\sum_{j=9}^{\infty}\frac{1}{F_{j}+1}\\
                      & < 1.479 + \sum_{j=9}^{\infty}\frac{1}{F_{j}} \leq 1.479+\sum_{j=9}^{\infty}\frac{1}{\phi^{j-2}}\approx1.535<2,
\end{align*}
and the proof is completed.     $\hfill\square$

Since the sum of the eigenvalues is equal to the trace, Theorem \ref{L1} yields the following substantial improvement of Theorem \ref{Bolzano}.

\begin{corollary} \label{final}
If $n \geq 10$ and $\lambda_1 < \lambda_2 < \cdots < \lambda_n$ are the $n$ simple eigenvalues of $F_R (n)$, then
\begin{align*}
  & -1 < \lambda_1 < 0 ,         \\
  &  1 < \lambda_2 < 2 ,         \\
  & F_i < \lambda_i < F_i + 2 , \quad  i=3,4,\ldots,w+1,  \\
  & F_i < \lambda_i < F_i + 1 , \quad  i=w+2,w+3,\ldots,n.
\end{align*}
\end{corollary}

\begin{figure}[h!]
  \centering
  \includegraphics[width=0.70\textwidth]{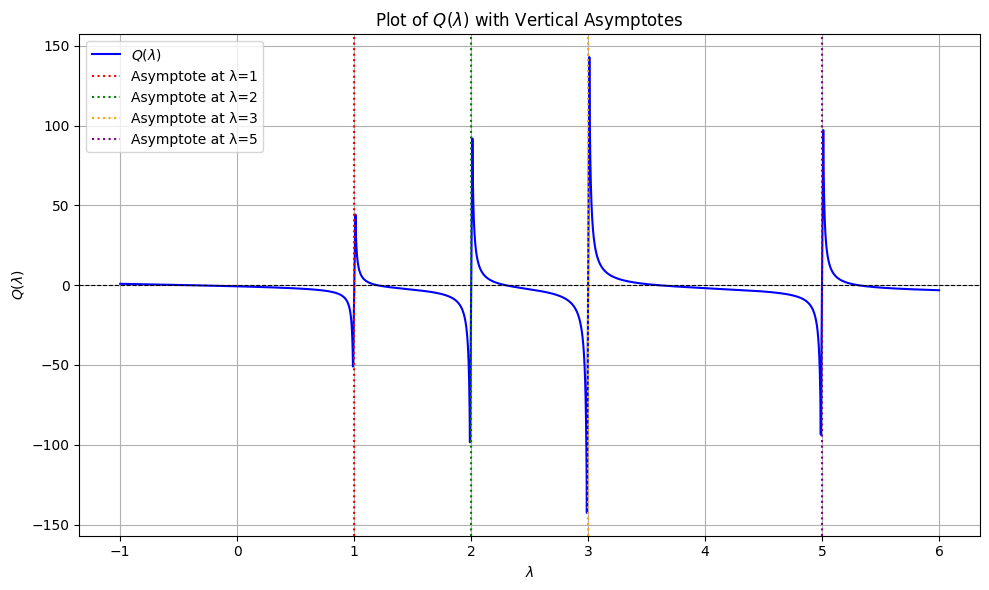}
  \caption{Plot of the function $Q_{F_{R}(5)} (z)$.
           The zeros of $Q_{F_{R}(5)} (z)$ are the eigenvalues of $F_R(5)$.
           The lines $\text{Re} (z) = 1 , 2 , 3 , 5$ are vertical asymptotes and ``close'' to the eigenvalues.}
\end{figure}

Figure 1 illustrates the behavior of the function $Q_{F_{R}(5)} (z)$ (i.e., $n=5$).
The zeros of this function are exactly the eigenvalues of $F_{R}(5)$.
The vertical asymptotes of the function, as well as, the solo negative eigenvalue are clearly visible.

\begin{remark}  \rm
As mentioning in the proof of Theorem \ref{secondtheorem},
\[
  \rho(F_R(n)) \sim  \frac{\phi^{n}}{\sqrt{5}} \quad \text{as $n \rightarrow \infty$.}
\]
By the well known (and easy to prove) identity of the Fibonacci numbers
$
    \sum\limits_{k=1}^{n} F_{k} = F_{n+2} - 1 ,
$
we get that
\[
  \text{tr} ( F_R (n) ) \sim  \frac{\phi^{n+2}}{\sqrt{5}} \quad \text{as $n \rightarrow \infty$.}
\]
\end{remark}

%%%%%%%%%%%%%%%%%%%%%%%%%%%%%%%%%%%%%%%%%%%%%%%%%%%%%%%%%%%%%%%%%%%%%%%%%%%%%%%%%%%%%%%%%%%%%%%%%%%%%%%%%%%%%%%%%%
\section{Generalizations}  \label{section4}

In this section, we consider Redheffer--type matrices with entries coming from a sequence $\alpha$.
In particular, for a \textit{sequence of positive terms} $\alpha = \bigl ( a_j \bigr )_{j=1}^{n}$, with $n \geq 3$,
we introduce the $n \times n$ \textit{Redheffer--type matrix} $A_R (n) = \bigl [ A_{R}(i,j) \bigr ]$ such that,
for $i,j=1,2,\ldots,n$,
\begin{equation}
  A_{R}(i,j): = \begin{cases}
          1,     & \text{if $j=1$ and $i \neq 1$;} \\
          a_{i}, & \text{if $i\mid j$;} \\
          0,     & \text{otherwise.}                    \label{20}
       \end{cases}
\end{equation}
For example, for $n=8$,
\[
  A_R(8)=\begin{bmatrix}
   a_1 & a_1 & a_1 & a_1 & a_1& a_1 & a_1 & a_1 \\
   1   & a_2 & 0   & a_2 & 0  & a_2 & 0   & a_2 \\
   1   & 0   & a_3 & 0   & 0  & a_3 & 0   & 0   \\
   1   & 0   & 0   & a_4 & 0  & 0   & 0   & a_4 \\
   1   & 0   & 0   & 0   & a_5& 0   & 0   & 0   \\
   1   & 0   & 0   & 0   & 0  & a_6 & 0   & 0   \\
   1   & 0   & 0   & 0   & 0  & 0   & a_7 & 0   \\
   1   & 0   & 0   & 0   & 0  & 0   & 0   & a_8
   \end{bmatrix} .
\]

Here, we will mainly study the asymptotic behaviour of the determinant of $A_R (n)$ as $n\rightarrow \infty$.
Following the steps of Section \ref{section2}, we can obtain the following result.

\begin{theorem}  \label{TH20}
Consider the Redheffer--type matrix $A_{R}(n)$ defined by (\ref{20}) corresponding to a sequence $\alpha = \bigl ( a_j \bigr )_{j=1}^{n}$ of positive terms.
Then, we have
\begin{equation}\label{310}
  \det(A_R (n)) = \biggl ( \prod_{j=1}^{n} a_j \biggr ) \biggl [ 1 + \sum_{k=2}^{n}\frac{\mu(k)}{a_{k} } \biggr ] .
\end{equation}
\end{theorem}

In the general case, and for any fixed $n$, there are sequences $\alpha = \bigl ( a_j \bigr )_{j=1}^{n}$ (easy to construct) such that the matrix $A_R (n) $ is singular.
When $\det (A_R (n)) \neq 0$, is possible that all eigenvalues are positive or not.
There are even cases where a diagonal entry is itself an eigenvalue.
For example, if we consider the (finite) sequences $\alpha: = \bigl ( a_j \bigr )_{j=1}^{j=6} = ( 2, 3, 4, 5, 6, 7 )$
and $\beta := \bigl ( b_j \bigr )_{j=1}^{j=6} = \bigl ( \frac{1}{2}, 1, 3, 4, 2, \sqrt{3} \bigr )$
and the corresponding Redheffer--type matrices $A_R (6)$ and $B_R (6)$, then
\[
    \sigma(A_R (6)) = \{ 0.77339, 2.31234, 3.94591, 5.41018, 6.21976, 8.33843 \}
\]
and
\[
\sigma(B_R(6))=\{-0.12466,\;0.79118,\;\sqrt{3},\;2.11914,\;3.43164,\;4.28271\}.
\]
So, $A_R (6)$ has all its eigenvalues positive, while $B_R (6)$ has a negative eigenvalue and an eigenvalue equal to a diagonal entry.
It is worth mentioning that any eigenvector of the matrix $B_R (6)$ corresponding to the simple eigenvalue $\sqrt{3}$ (which is a diagonal entry of $B_R (6)$) is of the form
$(0,(1+\sqrt{3})t,(-3-\sqrt{3})t,0,0,2t)^{T}$ $(t \neq 0)$, and the sum of its entries is zero (see Lemma \ref{sav}).

Furthermore, there are cases where a Redheffer--type matrix has multiple real eigenvalues or complex eigenvalues.
For example, if we consider the (strictly) increasing sequence $\gamma := \bigl ( c_j \bigr )_{j=1}^{j=4} = ( a, 7, 8, 9 )$ $( 0 < a < 7)$
and the corresponding Redheffer--type matrix $C_R (4)$, then it is easy to see that for $a\approx6.38002875442567$,
\[
    \sigma(C_R (4)) = \{ 4.77497, 4.77497, 8.13053, 12.69956 \} .
\]
In this special case, $C_R (4)$ has a double real eigenvalue. Moreover, as the parameter $a$ increases beyond the value $6.380028754425$, the double eigenvalue $\lambda_{0}\approx4.7749725154$
splits into a pair of complex conjugate simple eigenvalues; this behaviour, in particular, confirms the existence of the double eigenvalue.

Concerning the asymptotics of the determinant of $A_R (n)$, as $n \rightarrow \infty$, the problem may be treated as two separate problems, namely,
the behaviour of the product of the terms of the sequence $\alpha$ and the asymptotics of the sum $\sum\limits_{k=1}^{n}\dfrac{\mu(k)}{a_{k}}$.
Next, we will illustrate two intriguing examples.

\begin{example} \rm
(Power law sequence)
Let $a_j = j^p$, with $p>1$. In this case, $a_{1}=1$ and $\mu(1)=1$. Hence, (\ref{310}) yields
\begin{equation}\label{310b}
     \det (A_R (n)) = \left ( \prod_{j=1}^{n} a_j \right ) \sum_{k=1}^{n} \frac{\mu(k)}{a_{k}}.
\end{equation}
By Stirling's formula \cite{KPG}, and as $n \rightarrow \infty$, we have
$\prod\limits_{j=1}^{n}a_j = ( n! )^p \sim ( 2n \pi )^{\frac{p}{2}}  \biggl ( \dfrac{n}{e} \biggr )^{np}$.
On the other hand,
\begin{equation}
     \sum_{k=1}^{n}\frac{\mu(k)}{k^p}\sim \frac{1}{\zeta(p)},     \label{inverse}
\end{equation}
where we have used that the Dirichlet series, which generates the M$\ddot o$bius function, is the (multiplicative) inverse of the {\it Riemann zeta function} $\zeta(p)$, as long as $p>1$.
Hence, as $n \rightarrow \infty$,
\begin{equation*}
      \det (A_R (n)) \sim \zeta(p)^{-1} ( 2n \pi )^{\frac{p}{2}}  \biggl ( \frac{n}{e} \biggr )^{np} .
\end{equation*}
Notice that in the case $p=2$, the corresponding sum of (\ref{inverse}) plays a central role in the average order of the {\it Euler's totient function} $\phi(n)$ (see, e.g., \cite{APOSTOL} for details).
We would also like to remind to the reader that, if $p=1$ (which, however, is not the case here), then
$\sum\limits_{k=1}^{n}\dfrac{\mu(k)}{k} = \dfrac{1}{1} - \dfrac{1}{2} - \dfrac{1}{3} - \dfrac{1}{5} + \dfrac{1}{6} - \dfrac{1}{7} + \cdots = o (1)$, as $n\rightarrow \infty$.
This is well known that implies the prime number theorem, i.e., $\pi(n) \sim \dfrac{n}{\log n}$ as $n\rightarrow \infty$, where $\pi(n)$ is the {\it prime counting function}.
\end{example}

\begin{example} \rm
(Exponential sequence)
Let $q>1$, and consider the sequence $\alpha = \bigl ( a_j \bigr )_{j=1}^{\infty}$, with $a_j=q^j-1$ $(j=1,2,\ldots)$,
and the $n \times n$ \textit{Redheffer--type matrix} $A_R (n) = \bigl [ A_{R}(i,j) \bigr ]$ such that, for $i,j=1,2,\ldots,n$,
\[
A_R(i,j) =
\begin{cases}
1,     & \text{if $j=1$ and $i\neq 1$;}\\
q^i-1, & \text{if $i\mid j$;}\\
0,     & \text{otherwise}.
\end{cases}
\]
By Theorem \ref{TH20}, we have
\begin{equation}
\det (A_R(n))
=
\left(\prod_{j=1}^{n}(q^j-1)\right)
\left(
1+\sum_{k=2}^{n}
\frac{\mu(k)}{q^k-1}
\right).
\label{46}
\end{equation}
We observe that
\begin{equation}  \label{47}
\sum_{k=1}^{\infty} \frac{\mu(k)}{q^k-1}
=
\sum_{k=1}^{\infty} \mu(k)\sum_{m=1}^{\infty}q^{-km}
=
\sum_{r=1}^{\infty}
q^{-r}\sum_{k\mid r}\mu(k)
=
\frac{1}{q},
\end{equation}
where the interchange of summation is justified by absolute convergence, while the last equality follows from the identity
\[
\sum_{k\mid r}\mu(k)
=
\begin{cases}
1, & \text{if $r=1$;}\\
0, & \text{if $r\geq 2$.}
\end{cases}
\]
Since $\mu(1)=1$, relation (\ref{47}) yields
\begin{equation}  \label{48}
1+\sum_{k=2}^{n}
\frac{\mu(k)}{q^k-1}
=
1+\frac1q-\frac1{q-1}
-\sum_{k=n+1}^{\infty}
\frac{\mu(k)}{q^k-1}
=
\frac{q^2-q-1}{q(q-1)}
+\mathcal{O}_q(q^{-n}),
\end{equation}
where we have used the estimate
\[
\left|
\sum_{k=n+1}^{\infty}
\frac{\mu(k)}{q^k-1}
\right|
\leq
\sum_{k=n+1}^{\infty}
\frac1{q^k-1}
\leq
\frac{q^{-(n+1)}}{(1-q^{-1})^2}.
\]
Notice that, in (\ref{48}) and throughout, the implied constants may depend on the fixed parameter $q>1$. We also have
$\prod\limits_{j=1}^{n}(q^j-1) = \prod\limits_{j=1}^{n}q^j(1-q^{-j}) = q^{n(n+1)/2} \prod\limits_{j=1}^{n}(1-q^{-j})$.
Since $\sum_{j=1}^{\infty}q^{-j}<\infty$, the infinite product $C_q:=\prod\limits_{j=1}^{\infty}(1-q^{-j})$ converges to a positive number.
Hence,
\begin{equation}  \label{49}
\prod_{j=1}^{n}(q^j-1)
=
C_q\,q^{n(n+1)/2}
\left(1+\mathcal{O}_q(q^{-n})\right).
\end{equation}
By relations (\ref{46}), (\ref{48}), and (\ref{49}), we obtain
\[
\det A_R(n)
=
C_q\,q^{n(n+1)/2}\, \left(
\frac{q^2-q-1}{q(q-1)}
+\mathcal{O}_q(q^{-n})
\right).
\]
Now, in the case where $q \neq \varphi := \dfrac{1+\sqrt{5}}{2}$ (and thus, $q^2 - q - 1 \neq 0$), it follows that
\begin{equation}  \label{50}
\det A_R(n)
\sim
\frac{q^2-q-1}{q(q-1)}
\,C_q\,q^{n(n+1)/2}
\quad \text{as $n\rightarrow \infty$.}
\end{equation}
Since $C_q>0$, relation (\ref{50}) implies that, for $q\in(1,\varphi)$, the quantity $\det A_R(n)$ is negative
for all sufficiently large $n$, while for $q>\varphi$, it is positive for all sufficiently large $n$.\\
If $q=\varphi$, then relation (\ref{46}) yields $1+\sum\limits_{k=2}^{n} \frac{\mu(k)}{\varphi^k-1}=\mathcal{O}(\varphi^{-n})$.
By invoking relations (\ref{48}) and (\ref{50}), we get
\[
\det A_R(n)
=
\mathcal{O}\left(\varphi^{n(n-1)/2}\right)
\quad \text{as $n\rightarrow \infty$.}
\]
%%%%%%%%%%%%%%%%%%%%%%%%%%%%%%%%%%%%%%%%%%%%%%%%%%%%%%%%%%%%%%%%%%%%%%%
% In the special case $q=2$, we have
% \begin{equation}  \label{54}
% \det (A_R(n))
% \sim
% \left(
% \prod_{j=1}^{\infty}(1-2^{-j})
% \right)
% 2^{\frac{\left(n-1\right)\left(n+2\right)}{2}}
% \quad \text{as $n\rightarrow \infty$.}
% \end{equation}
% From the celebrated Euler's pentagonal-number formula, namely,
% \[
% \prod_{n=1}^{\infty}(1-x^n)
% =1+
% \sum_{k=1}^{\infty}
% (-1)^k \left ( x^{\omega(k)}+x^{\omega(-k)} \right ) ,
% \quad |x|<1,
% \]
% where
% \[
% \omega(l)=\frac{3l^{2}-l}{2},\quad l=0,\pm 1, \pm 2,\ldots
% \]
% (see Theorem 14.3 in \cite{APOSTOL}), it follows
% \[
% \prod_{j=1}^{\infty}\left(1-2^{-j}\right)
% =0.288788095086602421278899721929 \cdots .
% \]
%%%%%%%%%%%%%%%%%%%%%%%%%%%%%%%%%%%%%%%%%%%%%%%%%%%%%%%%%%%%%%%%%%%%%%%%%%%%
\end{example}

\begin{example}  \rm
Finally, let us consider the following variant of the Fibonacci--Redheffer matrix $F_R (n)$ (see (\ref{2})).
Suppose that all the entries of the matrix are the same except (and this is the only difference) the $(1,1)$-th entry, which in some sense, appears to be the most interesting entry of the matrix.
In particular, we set $a_{1,1} := 1 + b$ for some $b > -1$,
denote the new matrix by $\tilde{ F_R} (n)$, and ask the following question: \textit{are there values of $b$, for which the matrix $\tilde{ F_R} (n)$ is singular?}
By mimicking our approach of Section \ref{section2}, we can prove that
\[
  \det(\tilde{ F_R} (n)) = \prod_{j=1}^{n}\lambda_j = n!_{F} \biggl ( b + \sum_{k=1}^{n} \frac{\mu(k)}{F_{k}} \biggr ) ,
\]
where $n!_{F}$ is, again, the Fibonacci factorial. Hence, the answer to the above question is \textit{yes, if}
\[
           b = -\sum_{k=1}^{n}\frac{\mu(k)}{F_{k}},\quad n\geq2 .
\]
Clearly, if $n\rightarrow \infty$, then in the limit, we have $b=-C \approx 0.64572472\cdots$ (see Theorem \ref{secondtheorem}).
\end{example}

%%%%%%%%%%%%%%%%%%%%%%%%%%%%%%%%%%%%%%%%%%%%%%%%%%%%%%%%%%%%%%%%%%%%%%%%%%%%%%%%%%%%%%%%%%%%%%%%%%%%%%%%%%%%%%%%%%%%%%%%%%%%%%%%%%%%%%%%%%%

%%%%%%%%%%%%%%%%%%%%%%%%%%%%%%%%%%%%%%%%%%%%%%%%%%%%%%%%%%%%%%%%%%%%%%%%%%%%%%%%%%%%%%%%%%%%%%%%%%%%%%%%%%%%%%%%%%%%
\section{Acknowledgements} \label{section5}

The authors gratefully acknowledge Professor Dimitrios Koukoulopoulos for his crucial contribution to the proof of Theorem \ref{secondtheorem},
as well as Professor Sergey Savchenko for stimulating communication that led to Lemma \ref{sav}.
The first author has been partially supported by project MIS 5154714 of the National Recovery and Resilience Plan Greece 2.0 funded by the European Union under the NextGenerationEU Program.

%%%%%%%%%%%%%%%%%%%%%%%%%%%%%%%%%%%%%%%%%%%%%%%%%%%%%%%%%%%%%%%%%%%%%%%%%%%%%%%%%%%%%%%%%%%%%%

%%%%%%%%%%%%%%%%%%%%%%%%%%%%%%%%%%%%%%%%%%%%%%%%%%%%%%%%%%%%%%%%%%%%%%%%%%%%%%%%%%%%%%%%%%%%%

%%%%%%%%%%%%%%%%%%%%%%%%%%%%%%%%%%%%%%%%%%%%%%%%%%%%%%%%%%%%%%%%%%%%%%%%%%%%%%%%%%%%%%%%%%%%%
\end{document}